\documentclass[12pt,reqno]{amsart}
\usepackage{hyperref}
\usepackage{amsmath,amsthm,amsopn,amssymb,a4wide,varioref}
\usepackage{enumitem}
\usepackage[mathscr]{eucal}
\usepackage{tikz}
\usepackage{color, bm, amscd, tikz-cd}
\newcommand{\bydef}{\stackrel{\rm def}{=}}

\newcommand{\cB}{{\mathcal B}}

\newcommand{\cE}{{\mathcal E}}

\newcommand{\cH}{{\mathcal H}}

\newcommand{\cK}{{\mathcal K}}

\newcommand{\cM}{{\mathcal M}}

\newcommand{\cT}{{\mathcal T}}

\newcommand{\bC}{{\mathbb C}}

\newcommand{\fJ}{\mathfrak J}

\newcommand{\fT}{\mathfrak T}

\newtheorem{thm}{Theorem}[section]

\newtheorem{corollary}[thm]{Corollary}
\newtheorem{lemma}[thm]{Lemma}

\newtheorem{theorem}{Theorem}

\theoremstyle{definition}

\newtheorem{definition}[thm]{Definition}
\newtheorem{remark}[thm]{Remark}

\numberwithin{equation}{section}

\def\textmatrix#1&#2\\#3&#4\\{\bigl({#1 \atop #3}\ {#2 \atop #4}\bigr)}
\def\dispmatrix#1&#2\\#3&#4\\{\left({#1 \atop #3}\ {#2 \atop #4}\right)}
\numberwithin{equation}{section}

\def\textmatrix#1&#2\\#3&#4\\{\bigl({#1 \atop #3}\ {#2 \atop #4}\bigr)}
\def\dispmatrix#1&#2\\#3&#4\\{\left({#1 \atop #3}\ {#2 \atop #4}\right)}

\begin{document}

\title[Polydisc Toeplitz operators]{Toeplitz operators and Hilbert modules on the symmetrized polydisc}

\author[Bhattacharyya]{Tirthankar Bhattacharyya}
\address[Bhattacharyya]{Department of Mathematics, Indian Institute of Science, Bangalore 560 012, India.}
\email{tirtha@iisc.ac.in}

\author[Das]{B. Krishna Das}
\address[Das]{Department of Mathematics, Indian Institute of Technology Bombay, Powai, Mumbai 400076, India.}
\email{dasb@math.iitb.ac.in; bata436@gmail.com}

\author[Sau]{Haripada Sau}
\address[Sau]{Department of Mathematics, Indian Institute of Science Education and Research, Pashan, Pune 411008, India.}
\email{haripadasau215@gmail.com; hsau@iiserpune.ac.in}


\begin{abstract}
When is the collection of $\mathsf S$-Toeplitz operators with respect to a tuple of commuting bounded operators $\mathsf S= (S_1, S_2, \ldots , S_{d-1}, P)$, which has the symmetrized polydisc as a spectral set, non-trivial? The answer is in terms of powers of $P$ as well as in terms of a unitary extension. En route, Brown-Halmos relations are investigated. A commutant lifting theorem is established. Finally, we establish a general result connecting the $C^*$-algebra generated by the commutant of $\mathsf S$ and the commutant of its unitary extension $\mathsf R$.
\end{abstract}
\maketitle

\section{Introduction}
A famous result of Brown and Halmos shows that a bounded linear operator $A$ on the Hardy space $H^2$ is a Toeplitz operator if and only if $M_z^*AM_z = A$, where $M_z$ denotes the unilateral shift. This led to the study of operators $A$ on general Hilbert spaces satisfying an operator equation of the form $T^*AT = A$ by many authors.

Recently, the study of Toeplitz operators has been extended to the Hardy space of the symmetrized bidisc, see \cite{BDS}. This motivates our study. The characterizations obtained in this note are in terms of Hilbert modules which is a natural settings for operator theory in several variables. Completely positive maps were used by Prunaru in \cite{PrunaruJFA} because they were well-suited for the Euclidean unit ball. We use both. It is this interplay of completely positive maps with Hilbert modules which allows the results to be presented in their most general and natural form.

Let $\mathbb D$ be the open unit disc while $\mathbb D^d$, $\overline{\mathbb D}^d$ and $\mathbb T^d$ denote the open polydisc, the closed polydisc, and the $d$-torus, respectively in $d$-dimensional complex space for $d\ge 2$.

If $T_1, T_2, \ldots , T_d$ are commuting contractions on a Hilbert space $\mathcal H$ and $P$ denotes the product $T_1 T_2 \ldots T_d$, then an application of the arguments in \cite{Muhly} shows that there is a non-trivial bounded operator $A$ on $\mathcal H$ satisfying $T_i^*AT_i = A$ for all $i=1,2, \ldots , d$ if and only if $P^n$ does not converge to the zero operator strongly. This may remind an astute reader of Corollary 2.2 in Prunaru (\cite{PrunaruJFA}) where he showed that a completely positive, completely contractive and ultraweakly continuous linear map on $\mathcal B ( \mathcal H)$ has a non-zero operator as a fixed point if and only if its powers at the identity operator converge strongly to the zero operator. Prunaru's results were attuned to the Euclidean unit ball whereas the case of commuting contractions is attuned to the polydisc. In this note, we discuss a third domain which has gained prominence over the last two decades.

Consider the elementary symmetric polynomials $e_k$ for $k=1,2, \ldots, d$ in $d$ variables:
$$ e_k(z_1, z_2, \ldots , z_d) = \sum_{1 \le j_1 < j_2 < \cdots <j_k \le d} z_{j_1} z_{j_2} \ldots z_{j_k}.$$
By convention, $e_0$ is the constant polynomial $1$. The closed symmetrized polydisc is the polynomially convex set
$$ \Gamma_d = \{ (e_1(\bm z), e_2(\bm z), \ldots , e_d(\bm z): \bm z=(z_1, z_2, \ldots , z_d) \in \overline{\mathbb D}^d\}.$$
A point of $\Gamma_d$ will usually be denoted by $(s_1, s_2, \ldots , p)$. The number $d \ge 2$ will be fixed for this paper.

Let $\mathsf S=(S_1,\dots,S_{d-1},P)$ be a  commuting $d$-tuple of bounded operators on a Hilbert space $\mathcal H$.  Let $\mathcal A = \mathbb C [z_1, z_2, \ldots , z_d]$ be the algebra of polynomials in $d$ commuting variables. Consider the $\mathcal A$-module structure induced on $\mathcal H$ as follows:
$$ f \cdot h = f(S_1,\dots,S_{d-1},P)h \text{ for } f \in \mathcal A \text{ and } h \in \mathcal H.$$
Let $\|f\|$ be the supremum norm of $f$ over $ \Gamma_d$. The $\mathcal A$-module $\mathcal H$ is called a {\em contractive Hilbert module} if
\begin{equation} \label{CM} \|f(S_1,\dots,S_{d-1},P)h\| \le \|f\|\|h\| \text{ for } f \in \mathcal A \text{ and } h \in \mathcal H.\end{equation}
 The polynomial convexity of $ \Gamma_d$ implies that \eqref{CM} holds for every function holomorphic in a neighbourhood of $ \Gamma_d$. We could call such an $\mathcal H$ a $\Gamma_d$-contractive Hilbert module because  $\Gamma_d$ is a spectral set for $\mathsf S$. However, in this note, the set $\Gamma_d$ is fixed. Hence, the brevity. The commuting tuple $\mathsf S$ satisfying \eqref{CM} is also known as a $\Gamma_d$-contraction, see \cite{BSR}.
 We shall write a Hilbert module as above as $(\cH, \mathsf S)$ because we shall often need the tuple $\mathsf S$. For an account of why the language of Hilbert modules is a natural one for non-trivial extension of operator theory results to several variables, see \cite{DP}. In the present context, see \cite{Sarkar}.

If $(\mathcal K , \mathsf R)$ is a contractive module as above where $\mathsf R = (R_1, \ldots , R_{d-1}, U)$ consists of normal operators and the Taylor joint spectrum $\sigma(\mathsf R)$ is contained in the distinguished boundary $b\Gamma_d = \{ (s_1,  \ldots , s_{d-1}, p) : |p| = 1\}$ (see Theorem 2.4 (iii) in \cite{BSR}), $(\mathcal K, \mathsf R)$ will be called a {\em unitary Hilbert module}. Let $(\mathcal M, \mathsf T)$ with $\mathsf T = (T_1, \ldots , T_{d-1}, V)$ be a submodule of a unitary Hilbert module $(\mathcal K , \mathsf R)$. Then it is called an {\em isometric Hilbert module}. The operator tuples $\mathsf R$ and $\mathsf T$ are called a $\Gamma_d$-unitary and a $\Gamma_d$-isometry respectively.

If $\mathsf S^* \bydef (S_1^*,\dots,S_{d-1}^*,P^*)$, then {\em the adjoint module} $(\cH, \mathsf S^*)$ is a contractive module when $(\cH, \mathsf S)$ is so. This is easy to see from the definition. Clearly, the adjoint module of a unitary module is again a unitary module. Obviously, the adjoint of an isometric module need not be an isometric module.

Given two Hilbert modules $(\cH, \mathsf S)$ and $(\cK, \mathsf R)$, a {\em Hilbert module homomorphism} $\fJ: \cH\to \cK$ is a bounded operator $\fJ$ such that
$$ \fJ ( f \cdot h ) = f \cdot \fJ h \text{ for all } f \in \mathcal A \text{ and } h \in \mathcal H.$$
Such a homomorphism will often be called a {\em module map}. If, moreover, $\fJ$ is a contraction, we say that the module $(\cH, \mathsf S)$ is contractively embedded as a submodule of $(\cK, \mathsf R)$. A contractive module map $\fJ : \cH \to \cK$ is called a {\em canonical module map} if $(\cK, \mathsf R)$ is {\em minimal in the sense that there is no submodule of $(\cK, \mathsf R)$ containing $(\cH, \mathsf S)$ and reducing $\mathsf R$} and
\begin{align}\label{ContEmbedd}
\fJ^*\fJ=\operatorname{SOT-}\lim P^{*n}P^n.
\end{align}

\begin{definition}
Let $(\mathcal H, \mathsf S)$ be a contractive Hilbert module. A bounded operator $A$ on $\mathcal H$ is said to be an {\em $\mathsf S$-Toeplitz operator} if it satisfies the {\em Brown-Halmos relations} with respect to $\mathsf S$:
\begin{eqnarray}\label{Gen-Brown-Halmos-Poly}
S_i^*AP=AS_{d-i} \text{ for each $1\leq i \leq d-1$ and } P^*AP = P.
\end{eqnarray}
The $*$--closed and norm closed vector space of all $\mathsf S$-Toeplitz operators is denoted by $\mathcal T(\mathsf S)$.
\end{definition}

See \cite{BDS} and \cite{DS} for the motivation of the definition above.

\begin{theorem}\label{Thm:Ext} Let $(\cH, \mathsf S)$ be a contractive Hilbert module. Then the following are equivalent.

 \begin{enumerate}
 \item $\mathcal T(\mathsf S)$ is non-trivial.
 \item $Q=\operatorname{SOT-}\lim P^{*n}P^n \neq 0$.
 \item $(\mathcal H, \mathsf S)$ can be canonically embedded as a submodule of a unitary module $(\mathcal K, \mathsf R)$, which is unique up to unitary isomorphism.
 \item $(\cH, \mathsf S)$ can be contractively embedded as a submodule of an isometric Hilbert module.

\end{enumerate}

Moreover, in this case, for any unitary module $(\cK',\mathsf R')$ and any contractive embedding $\fJ' : \cH \rightarrow \cK'$, the following are true.
\begin{enumerate}[label=(\roman*)]
\item $\fJ'^*\fJ'\leq\operatorname{SOT-}\lim P^{*n}P^n=\fJ^*\fJ$,

\item $(\cK, \mathsf R)$ is contractively embedded into $(\cK', \mathsf R')$ through a contraction $\fT:\cK\to\cK'$ such that $\fT\fJ=\fJ'$.
\end{enumerate}

\end{theorem}

This theorem is proved in section 2.  In Theorem \ref{Thm:Ext} above, $\mathcal T(\mathsf S)$ depends on the $S_i$ and $P$ whereas the condition (2) is in terms of $P$ alone. That is the surprising strength.

 The {\em Toeplitz $C^*$-algebra}, denoted by $C^*(I_{\cH},\cT(\mathsf S))$, is the $C^*$-algebra generated by $I_{\cH}$ and the vector space $\cT(\mathsf S)$ of $\mathsf S$-Toeplitz operators.

\begin{theorem}\label{Thm:ComLftPDisc}
Let $(\cH, \mathsf S)$ be a contractive Hilbert module satisfying
$$Q=\operatorname{SOT-}\lim P^{*n}P^n \neq 0.$$
Then  $(\cH, \mathsf S)$ can be canonically embedded as a submodule of a unitary module $(\cK,\mathsf R)$ by a module map $\fJ$ such that
\begin{enumerate}
\item the map $\rho$ defined on the commutant algebra $\{R_1,\dots,R_{d-1},U\}'$  by $$\rho(Y)=\fJ ^*Y\fJ ,$$ is a complete isometry onto $\mathcal T(\mathsf S)$;
\item there exists a surjective unital $*$-representation $$\pi:\mathcal C^*\{I_{\mathcal H}, \mathcal T(\mathsf S)\}\to \{R_1,\dots,R_{d-1},U\}'$$ such that $\pi \circ \rho =I;$
\item there exists a completely contractive, unital and multiplicative mapping
$$\Theta:\{S_1,\dots,S_{d-1},P\}'\to \{R_1,\dots,R_{d-1},U\}'$$
defined by $\Theta(X)=\pi(\fJ ^*\fJ X)$ which satisfies $$\Theta(X)\fJ =\fJ X.$$
\end{enumerate}
\end{theorem}

When the contractive module above is an isometric one, the following stronger version holds.

\begin{theorem}\label{Thm:Gamma_n isometry}
Let $(\mathcal H , \mathsf S)$ be an isometric Hilbert module.
\begin{enumerate}
\item[(1)] There exists a unitary module $ (\mathcal K, \mathsf R )$ containing $(\mathcal H , \mathsf S)$ as an isometrically embedded submodule such that $\mathsf R$ is the minimal extension of
$\mathsf S$. In fact, $\mathcal K$ is the span closure of the following elements:
$$\{U^{ m}h:h\in\mathcal H, \text{ and } m \in \mathbb Z\}.
$$
    Moreover, any operator $X$ acting on $\mathcal H$ commutes with $\mathsf S$ if and only if
    $X$ has a unique norm preserving extension $Y$ acting on $\mathcal K$ commuting with $\mathsf R$.
\item[(2)] An operator $X$ is in $\mathcal T(\mathsf S)$ if and only if there exists a unique operator $Y$ in the commutant of the von-Neumann algebra generated by
$\{R_1,\dots,R_{d-1},U\}$ such that $\|X\|=\|Y\|$ and
    $
    X=P_\mathcal HY|_{\mathcal H}.
    $
\item[(3)] Let $\mathcal C^*(\mathsf S)$ and $\mathcal C^*(\mathsf R)$ denote the unital $\mathcal C^*$-algebras generated by
$\mathsf S$ and $\mathsf R$ respectively and $\mathcal I(\mathsf S)$
denote the closed ideal of $\mathcal C^*(\mathsf S)$ generated by all the commutators $XY-YX$ for
 $X,Y\in \mathcal C^*(\mathsf S)\cap \mathcal T(\mathsf S)$.
Then there exists a short exact sequence
    $$
    0\rightarrow\mathcal I(\mathsf S)\hookrightarrow \mathcal C^*(\mathsf S)\xrightarrow{\pi_0} \mathcal C^*(\mathsf R)\rightarrow 0
    $$
    with a completely isometric cross section, where $\pi_0: \mathcal C^*(\mathsf S)\to \mathcal C^*(\mathsf R)$ is the canonical unital $*$-homomorphism which sends the generating set $\mathsf S$
    to the corresponding generating set $\mathsf R$, i.e., $\pi_0(P)=U$ and $\pi_0(S_i)=R_i$ for all $1\leq i \leq d-1$.
\end{enumerate}
\end{theorem}

The proof of this is in Section 3. As an application of the above, we obtain a version of the commutant lifting theorem in Section 4.

\begin{theorem}\label{Thm:PLT}
Let $(\cH, \mathsf S)$ be a contractive Hilbert module. Let $\fJ : \cH \to \cK$ be a canonical embedding of $(\cH, \mathsf S)$ as a submodule of a unitary Hilbert module $(\cK,\mathsf R)$. Let $X$ be in the commutant of $\mathsf S$. Then the module $(\cH, X)$ can be contractively embedded as a submodule of $(\cK, Y)$ for an $Y$ in the commutant of $\mathsf R$ and  $\|Y\|\leq \|X\|$.
\end{theorem}

\section{Proof of Theorem \ref{Thm:Ext} and relation to quotient modules}

Fix a contractive Hilbert module $(\cH, \mathsf S)$ or, in other words, a $\Gamma_d$-contraction $\mathsf S=(S_1,\dots,S_{d-1},P)$. It is well-known that
$$S_i - S_{d-i}^*P = D_P F_i D_P; \;\; i=1,2, \ldots , d-1$$
for a certain operator tuple $(F_1, F_2, \ldots , F_{d-1})$. The operators $F_i$ are called the {\em fundamental operators}. The existence was discovered for $d=2$ in \cite{BhPSR}. For a proof for $d > 2$, see \cite{APal} or \cite{SouravArxiv}. Alternatively, it is easily deducible from Proposition 2.5(3) of \cite{SouravCAOT}.

\begin{lemma}\label{new-Gresult}
For each $i=1,2,\dots d-1$, we have
\begin{align}\label{Gresult}
{P^*}^j(S_{d-i}-S_i^*P)P^j\to0 \quad\text{strongly as}\quad j\to\infty.
\end{align}
\end{lemma}
\begin{proof}
For every $h \in \mathcal H$, we have
\begin{eqnarray*}
\|P^{* j}(S_{d-i}-S_i^*P)P^jh\|^2&=&\|P^{* j}(D_PF_{d-i}D_P)P^jh\|^2
\\&\leq& \|F_{d-i}\|^2\|D_PP^jh\|^2=\|F_{d-i}\|^2(\|P^jh\|^2-\|P^{j+1}h\|^2).
\end{eqnarray*}The proof follows because the last term in bracket converges to zero as $j\to\infty$. \end{proof}

To prove Theorem \ref{Thm:Ext}, we shall take the path $(1)\Rightarrow(2)\Rightarrow(3)\Rightarrow(4)\Rightarrow(1)$.

\begin{proof}[{\bf Proof of $(1)\Rightarrow(2)$:}]
Let there be a non-zero $\mathsf S$-Toeplitz operator $A$. This, in particular, implies that for all $n \geq 0$ we have $A = P^{*n}AP^n$ and hence $\| Ah \| \le \| A\| \| P^nh\| $ for every vector $h$. So, if $P^n$ strongly converges to $0$, then $A = 0$ which is a contradiction.

\noindent{\bf Proof of $(2)\Rightarrow(3)$:} Assume $P^n\nrightarrow0$ strongly. As $P$ is a contraction
$$
 I_{\mathcal{H}}\succeq P^*P\succeq  P^{*2}P^2\succeq \cdots\succeq {P^*}^nP^n\succeq \cdots\succeq 0.
$$ This guarantees a positive non-zero contraction $Q$ such that
\begin{align}\label{assymplimit}
Q=\operatorname{SOT-}\lim P^{*n}P^n.
\end{align}
Clearly,
$$P^*QP= Q.$$ Hence we can define an isometry $V : \overline{\operatorname{Ran}}Q \rightarrow \overline{\operatorname{Ran} }Q $ satisfying
\begin{align}\label{X}
V:Q^\frac{1}{2}h \mapsto Q^\frac{1}{2}Ph \text{ for each }h\in\cH.
\end{align}
We now prove that $Q$ is an $\mathsf S$-Toeplitz operator. Indeed,
\begin{align*}
S_i^*QP-QS_{d-i}=\lim_j (S_i^*P^{* j}P^jP-P^{* j}P^jS_{d-i})=\lim_jP^{* j}(S_i^*P-S_{d-i})P^j.
\end{align*}
By \eqref{Gresult}, the limit above is zero and hence $S_i^*QP=QS_{d-i}$. Define operators $T_j:\overline{\operatorname{Ran}}Q \rightarrow \overline{\operatorname{Ran}}Q$ for $j=1,2,\dots, d-1$ as
\begin{align}\label{Xj}
T_j:Q^\frac{1}{2}h \mapsto Q^\frac{1}{2}S_jh.
\end{align}
It is straightforward to see that each $T_j$ is well-defined. The tuple $(T_1, \ldots , T_{d-1}, V)$ is a commuting tuple too.
 The computation below establishes the identities $T_i=T_{d-i}^*V$ for each $i=1,2,\dots,d-1$:
\begin{align*}
\langle T_{d-i}^*VQ^\frac{1}{2}h,Q^\frac{1}{2}h' \rangle=
\langle Q^\frac{1}{2}Ph,Q^\frac{1}{2}S_{d-i}h' \rangle= \langle QS_ih,h'\rangle=\langle T_iQ^\frac{1}{2}h,Q^\frac{1}{2}h'\rangle,
\end{align*}where to obtain the third equality, we use the fact that $Q$ is an $\mathsf S$-Toeplitz operator. We have already noted that $V$ is an isometry. Thus by Theorem 4.12 of \cite{BSR}, all we need to show to conclude that $(T_1,\dots,T_{d-1},V)$ is a $\Gamma_d$-isometry is that the $d-1$ tuple $(\gamma_1T_1,\dots,\gamma_{d-1}T_{d-1})$ is a $\Gamma_{d-1}$-contraction, where $\gamma_j=(d-j)/d$ for each $j=1,2,\dots,d-1$. This readily follows from the identity
$$
\xi(\gamma_1T_1,\dots,\gamma_{d-1}T_{d-1})Q^\frac{1}{2}=\xi(\gamma_1S_1,\dots,\gamma_{d-1}S_{d-1})Q^\frac{1}{2}
$$for every polynomial $\xi$ in $\bC[z_1,z_2,\dots,z_{d-1}]$, and the fact that $(\gamma_1S_1,\dots,\gamma_{d-1}S_{d-1})$ is a $\Gamma_{d-1}$-contraction where $\gamma_j = (d-j)/d$.

Let $\mathsf R=(R_1,\dots,R_{d-1},U)$ acting on $\cK$ be a minimal $\Gamma_d$-unitary extension of the $\Gamma_d$-isometry $\mathsf T=(T_1,\dots,T_{d-1},V)$. Define a contraction $\fJ:\cH\to\cK$ as
\begin{align*}
\fJ :h\mapsto Q^\frac{1}{2}h \text{ for every } h \in \cH.
\end{align*}
This is a module homomorphism because
\begin{align}\label{Int}
R_j\fJ h=R_jQ^\frac{1}{2}h=T_jQ^\frac{1}{2}h=Q^\frac{1}{2}S_jh=\fJ S_jh.
\end{align}Similarly, $U\fJ=\fJ P$. Finally, by definition of $\fJ$ and $Q$, it follows that $\fJ^*\fJ$ is the limit of $P^{*n}P^n$ in the strong operator topology.

For the uniqueness part, consider
$$(\fJ,\cK,\mathsf R)=( R_1,\dots, R_{d-1},U)$$
as above and
$$(\fJ',\cK',{\mathsf R'})=( R_1',\dots, R_{d-1}', U')$$
with similar properties. Define the operator $\tau:\mathcal K \to \mathcal K'$ densely by
$$
\tau:f(\mathsf R, \mathsf R^*)\fJ h\mapsto f(\mathsf R', \mathsf R'^*) \fJ  h
$$
for every $h\in \mathcal H$ and polynomial $f$ in $\bm z$ and $\overline{\bm z}$. The map $\tau$ is surjective by minimality. Note that $\tau$ clearly satisfies $\tau\fJ=\fJ'$. We will be done if we can show that $\tau$ is an isometry. Let $f$ be a polynomial in $\bm z$ and $\overline{\bm z}$ and $\bar{f}f=\sum a_{\bm n,\bm m}\bm z^{\bm n}\overline{\bm z}^{\bm m}$. Then for every $h\in \cH$,
\begin{align}\label{UniqComp}
\notag \|f(\mathsf R, \mathsf R^*)\fJ h\|^2
 \notag & = \sum a_{\bm n,\bm m}\langle \fJ^*\mathsf R^{* \bm m}\mathsf R^{\bm n}\fJ h, h\rangle\\
\notag & = \sum a_{\bm n,\bm m}\langle \mathsf S^{* \bm m}\fJ^*\fJ\mathsf S^{\bm n}h, h\rangle\\
 & = \sum a_{\bm n,\bm m}\langle \mathsf S^{* \bm m} Q\mathsf S^{\bm n}h, h\rangle.
\end{align}
Since the last term depends on $\mathsf S$ only, $\tau$ is an isometry.

\noindent{\bf Proof of $(3)\Rightarrow(4)$:} Obvious.

\noindent{\bf Proof of $(4)\Rightarrow(1)$:} Let $(\mathcal M, \mathsf T)$ be the isometric module with $\mathsf T = (T_1, \ldots , T_{d-1}, V)$. In this case, we have
$$
P^*\fJ^*\fJ P=\fJ^*V^*V\fJ=\fJ^*\fJ
$$and
$$ S_{d-j}^* \fJ^*\fJ P = \fJ^* T_{d-j}^*V \fJ = \fJ^*T_j \fJ=\fJ^*\fJ S_j.$$
This proves that the non-zero operator $\fJ^* \fJ$ belongs to $\mathcal T(\mathsf S)$. This in particular establishes that (4) implies (1).

To complete the proof, first note that the proof of $(3)\Rightarrow(1)$ above implies that $\fJ'^*\fJ'$ is an $\mathsf S$-Toeplitz operator. In particular, $\fJ'^* \fJ'$ is in $\cT(P)$. This implies
 \begin{align*}
 \fJ'^*\fJ'=P^{*n}\fJ'^*\fJ' P^n \leq P^{* n}P^n \text{ for every } n.
 \end{align*}
This proves part (i).

For part (ii) we define the operator $\fT:\mathcal K \to \cK'$ densely by
$$
\fT:f(\mathsf R, \mathsf R^*)\fJ h\mapsto f(\mathsf R', \mathsf R'^*) \fJ'  h
$$
for every $h\in \mathcal H$ and polynomial $f$ in $\bm z$ and $\overline{\bm z}$. Using part (i), a similar computation as done in \eqref{UniqComp} yields
\begin{align*}
 \|f(\mathsf R', \mathsf R'^*)\fJ' h\|\leq  \|f(\mathsf R, \mathsf R^*)\fJ h\| \text{ for every }h\in\cH.
\end{align*}
This shows that $\fT$ is not only well-defined but also a contraction. Finally, it readily follows from the definition of $\fT$ that it intertwines $\mathsf R$ and $\mathsf R'$ and that $\fT\fJ=\fJ'$.
\end{proof}

\begin{remark}\label{R:Can-nonCan}
It is known that the minimal unitary (or isometric) dilation space of a contraction acting on a Hilbert space is always infinite dimensional even in the case of matrices. In contrast, if $\cH$ is finite dimensional, the isometric module $(T_1,\dots,T_{d-1}, V)$ defined in \eqref{Xj} acts on a finite dimensional space, viz., $\overline{\operatorname{Ran}}Q$ and hence is a $\Gamma_d$-unitary.

\end{remark}

Biswas and Shyam Roy showed in \cite{BSR} that an isometric Hilbert module $(\mathcal M, \mathsf T)$ always decomposes (up to unitary equivalence) as the direct sum of a pure isometric Hilbert module and a unitary Hilber module, i.e., $\mathcal M = H^2(\cE) \oplus \cK$ where $\cE$ is a Hilbert space, $T_i = M_{A_i + A_{n-i}^*z} \oplus R_i$ for $i=1,2, \ldots , d-1$ and $V = M_z \oplus U$ for certain $A_i \in \cB(\cE)$ and $(\cK,\mathsf R)$ is a unitary Hilbert module, called the {\em unitary part} of  $(\mathcal M, \mathsf T)$.

If a Hilbert module $(\cH, \mathsf S)$ is isomorphic to $(\mathcal M, \mathsf T)$ quotiented by a submodule $(\mathcal M', \mathsf T|_{\mathcal M'})$, then   $(\cH, \mathsf S)$ is said to be realized as a {\em quotient module} or equivalently $(\cH, \mathsf S)$ is said to have a {\em co-extension}  $(\mathcal M, \mathsf T)$ because there is a co-isometric module map $L : \mathcal M \rightarrow \cH$; see \cite{DMS}. The module $(\mathcal M, \mathsf T)$ is said to be {\em minimal} over $(\cH, \mathsf S)$ if there is no proper submodule of $(\mathcal M, \mathsf T)$ which contains $(\cH, \mathsf S)$ and reduces $\mathsf T$.

\begin{lemma} \label{non-canonical}
If a contractive Hilbert module $(\cH, \mathsf S)$ can be realized as a quotient module of a minimal isometric module $(\mathcal M, \mathsf T)$ which has a non-trivial unitary part $(\cK, \mathsf R)$, then there is a (not necessarily canonical) contractive embedding of the adjoint module $(\cH, \mathsf S^*)$ in $(\cK, \mathsf R)$.
\end{lemma}

\begin{proof}
The proof consists of defining $ \fJ:\mathcal H \to \cK$ as
$$
\fJ:h\to P_{\cK}h,\quad (h\in\mathcal H)
$$
where $P_\cK$ denotes the orthogonal projection of $\cM$ onto $\cK$.   \end{proof}

Lemma \ref{non-canonical} has an interesting, but not a very unexpected, consequence whose proof is easy and hence omitted.
\begin{corollary}\label{P:Aux}
If a pure  contractive Hilbert module $(\cH, \mathsf S)$ can be realized as a quotient module of an isometric module $(\mathcal M, \mathsf T)$ which is minimal over $(\cH, \mathsf S)$, then the unitary part in the Biswas - Shyam Roy decomposition of $(\mathcal M, \mathsf T)$ is absent.
\end{corollary}

\section{Algebraic structure of the Toeplitz $C^*$-algebra}

\subsection{Proof of Theorem \ref{Thm:ComLftPDisc}.} We begin with a few lemmas. The first lemma below gives us the existence of an important completely positive map. This is a particular case of Lemma 2.1 in \cite{PrunaruJFA}. The central idea of the proof goes back to Arveson, see Proposition 5.2 in \cite{Arveson-Nest}. For a subnormal operator tuple, in the multivariable situation, Eschmeier and Everard have proven a similar result by direct construction, see Section 3 of \cite{EE}.

\begin{lemma} \label{L:PJFA}
Let $P$ be a contraction on the Hilbert space $\mathcal H$. Then there exists a completely positive, completely contractive, idempotent linear map
$\Phi : \mathcal B(\mathcal H) \to \mathcal B(\mathcal H)$ such that $\operatorname{Ran}  \Phi = \mathcal T(P)$.
Moreover, if $ A,B \in \mathcal B(\mathcal H)$ satisfy $P^*(AXB) P = A P^*XPB$ for all $X \in \mathcal B(\mathcal H)$ then $\Phi(AXB) = A\Phi(X)B$. In addition,
$$\Phi(I_\mathcal H ) = Q = \lim_{n\to \infty} P^{*n}P^n$$
where the limit is in the strong operator topology.
\end{lemma}
Part of Theorem 3.1 of \cite{CE} gives us more properties of $\Phi$.

\begin{lemma}[Choi and Effros] \label{L:CE}
Let $\Phi :\mathcal B(\mathcal H) \to \mathcal B(\mathcal H)$ be a completely positive and completely contractive map
such that $\Phi\circ\Phi=\Phi$. Then for all $X$ and  $Y$ in $\mathcal B(\mathcal H)$ we have
\begin{align}\label{Identities}
\Phi(\Phi(X)Y)= \Phi ( X\Phi(Y)) = \Phi( \Phi(X) \Phi(Y)).
\end{align}
\end{lemma}

The last lemma that we need will play a crucial role. Its proof follows from Theorem 3.1 in \cite{PrunaruJFA}. For someone not familiar with \cite{PrunaruJFA}, this might create some opacity and hence the simplified proof in our context is supplied below.

\begin{lemma}\label{L:PJFA2}
  There is an idempotent, completely positive and completely contractive map $\Phi:\mathcal B(\mathcal H)\to \mathcal B(\mathcal H)$ such that
\begin{align}\label{P-Toep}
\operatorname{Ran}\Phi=\{X\in \mathcal B(\mathcal H):P^*XP=X\}=\mathcal T(P).
\end{align}
If $(\mathcal K, \pi, \fJ)$ is the minimal Stinespring dilation of the restriction of $\Phi$ to the unital $C^*$-algebra $C^*(I_{\mathcal H}, \mathcal T (P))$ generated by $\mathcal T (P)$, and if $Q =$ SOT-$\lim P^{*n}P^n$, then the following properties are satisfied.

\noindent (${\bf{ P_1}}$) $U:=\pi(QP)$ is a unitary operator. Moreover, $\fJ P=U\fJ $ and $\mathcal K$ is the smallest reducing subspace for $U$ containing $\fJ \mathcal H$.

\noindent (${\bf{ P_2}}$) The map $\rho:\{U\}' \to \mathcal T(P)$ defined by $\rho (Y)=\fJ ^*Y\fJ $, for all $Y\in \{U\}'$, is surjective and a complete isometry.

\noindent (${\bf{ P_3}}$) The Stinespring triple $(\cK,\pi,\fJ )$ satisfies $\pi\circ\rho=I$. In particular,
$$\pi(C^*(I_\mathcal H,\mathcal T(P)))=\{U\}'.$$

\noindent (${\bf{ P_4}}$) The linear map $\Theta : \{P\}' \to  \{U\}'$
defined by $\Theta(X) = \pi(QX)$ is completely contractive, unital and multiplicative.
\end{lemma}

\begin{proof}
We restrict $\Phi$ to $\mathcal C^*(I_\mathcal H, \mathcal T(P))$ and continue to call it $\Phi$.

   By definition of minimal Stinepring dilation,
\begin{equation}
\label{Stines}
\Phi(X) = \fJ^*\pi (X)\fJ \text{ for every $X\in \mathcal C^*(I_\mathcal H, \mathcal T(P))$}.
\end{equation}
Note that $Q=\Phi(I_\mathcal H)=\fJ^*\fJ=\operatorname{SOT-}\lim_{n\rightarrow \infty} P^{*n}P^n.$

 The kernel of $\Phi$ is an ideal in $C^*(I_\mathcal H, \mathcal T(P))$ by Lemma~\ref{L:CE} (when $\Phi$ is allowed as a map on whole of $\mathcal B (\mathcal H)$, its kernel may not be an ideal) and hence it follows from the construction of the minimal Stinespring dilation that
$\text{Ker }\!\Phi= \text{Ker }\!\pi$. Thus
\begin{equation} \label{pix=piphix} \pi(X)=\pi(\Phi(X)) \text{ for any } X\in C^*(I,\mathcal T(P)). \end{equation}
Since $\pi$ is a representation,we get $ U^*\pi(X)U = \pi(X) \text{ for any } X \in C^*(I,\mathcal T(P)).$

Since $\pi$ is unital, we get that $U$ is an isometry. If $P'$ is a projection in the weak* closure of $\pi(C^*(I,\mathcal T(P)))$, then we also have $U^*P'U=P'$ and $U^* P'^{\perp}U=P'^{\perp}$. This shows that $UP'=P'U$ and therefore $\pi(X)U=U\pi(X)$ for all $X\in C^*(I,\mathcal T(P))$. In particular, it follows that $U$ is a unitary and
$ \pi(C^*(I_\mathcal H,\mathcal T(P)))\subseteq\{U\}'. $

We now prove an identity which will be used in this proof as well as later. The identity is

\begin{equation} \label{thetaXVisVX} \pi(QX)\fJ =\fJ X \end{equation}
for any $X \in \mathcal B(\mathcal H)$ that commutes with $P$.
The proof of \eqref{thetaXVisVX} follows from two computations. For every $h, h^\prime \in \mathcal H$, we have
\begin{align*}
\langle \pi(QX)\fJ h,\fJ h^\prime \rangle&=\langle \fJ ^*\pi(QX)\fJ h,k\rangle\\
&=\langle \Phi(QX)h,h^\prime \rangle  
\\
&=\langle QXh,h^\prime \rangle \quad [\text{because } \mathcal T (P) \text{ is fixed by } \Phi]\\& =\langle \fJ Xh,\fJ h^\prime \rangle \end{align*}
showing that $P_{\overline{\operatorname{Ran}} \fJ }\pi(QX)\fJ  = \fJ X$. On the other hand,
\begin{align*}
\|\pi(QX)\fJ h\|^2 &=\langle \fJ ^*\pi(X^*Q^2X)\fJ h,h \rangle \\
&=\langle \Phi(X^* Q^2X)h,h \rangle \\
&=\langle X^* \Phi(Q^2)Xh,h \rangle \quad [\text{by Lemma \ref{L:PJFA}}]\\
&=\langle X^*QXh,h\rangle \quad[\text{by Lemma \ref{L:CE}}]\\
&=\| \fJ Xh \|^2.
\end{align*}
Hence, \eqref{thetaXVisVX}  is proved.

A trivial consequence of \eqref{thetaXVisVX} is that
 $U\fJ=\fJ P$. To complete the proof of ${\bf{ P_1}}$, it is required to establish that $\mathcal K$ is the smallest reducing subspace for $U$ containing $\fJ \mathcal H$. To that end, we consider a map $\delta$ from $\operatorname{Ran} \pi$ into $\mathcal T (P)$ given by $$\delta(\pi(X))=\fJ^*\pi(X)\fJ=\Phi(X) \text{ for all } X\in C^*(I, \mathcal T(P)).$$
It is injective because $\text{Ker }\!\Phi= \text{Ker }\!\pi$.

Since $\delta\circ \pi=\Phi$, we have $\delta\circ \pi$ to be idempotent and this coupled with the injectivity of $\delta$ gives us $\pi\circ\delta=I$ on $\pi\{C^*(I,\mathcal T(P))\}$. This immediately implies that $\delta$ is a complete isometry.

Let $\mathcal K_0\subseteq \mathcal K$ be the smallest reducing subspace for $U$ containing $\fJ \mathcal H$. Let $P_{\mathcal K_0}$ be the projection in $\mathcal B(\mathcal K)$ onto the space $\mathcal K_0$. Consider the vector space \[P_{\mathcal K_0}\{U\}'P_{\mathcal K_0}:=\{P_{\mathcal K_0}XP_{\mathcal K_0}: X\in\{U\}'\}= \{P_{\mathcal K_0}X|_{\mathcal K_0}\oplus 0_{\mathcal K_0^{\perp}}: X\in\{U\}'\}.  \]
and the map $\delta': P_{\mathcal K_0}\{U\}'P_{\mathcal K_0}\to \mathcal T(P)\subseteq \mathcal B(\mathcal H)$ defined by $X\mapsto \fJ^* X\fJ$. This is injective.

Indeed, it is easy to check that $\fJ^* X\fJ\in \mathcal T(P)$ for $X\in \{U\}'$. Now if $\fJ^*X\fJ=0$ for some $X\in \{U\}'$ then using the identity $\fJ P=U\fJ $, we get that
$$\langle X f(U, U^*)\fJ h, g(U,U^*)\fJ k\rangle =0$$
for any two variable polynomials $f$ and $g$ and $h,k\in\mathcal H$. This shows that $P_{\mathcal K_0}XP_{\mathcal K_0}=0$ and therefore, $\delta'$ is injective. For any $Y\in P_{\mathcal K_0}\{U\}'P_{\mathcal K_0}$,
\begin{align*}
 \delta'(P_{\mathcal K_0}\pi(\fJ^*Y\fJ)P_{\mathcal K_0}-Y)=\fJ^*\pi(\fJ^*Y\fJ)\fJ-\fJ^*Y\fJ=  \Phi(\fJ^*Y\fJ)-\fJ^*Y\fJ=0.
\end{align*}
Thus, by the injectivity of $\delta'$, we have $ P_{\mathcal K_0}\pi(C^*(I,\mathcal T(P)))P_{\mathcal K_0}=P_{\mathcal K_0}\{U\}'P_{\mathcal K_0}$.
In other words, we have a surjective complete contraction
\[ \tilde{C}_{\mathcal K_0}: \pi(C^*(I,\mathcal T(P))) \to P_{\mathcal K_0}\{U\}'P_{\mathcal K_0}= \{P_{\mathcal K_0}X|_{\mathcal K_0}\oplus 0_{\mathcal K_0^{\perp}}: X\in\{U\}'\},\]
defined by $X\mapsto P_{\mathcal K_0}X P_{\mathcal K_0}$. Since $\delta=\delta'\circ \tilde{C}_{\mathcal K_0}$ and $\delta$ is a complete isometry,
$\tilde{C}_{\mathcal K_0}$ is a complete isometry. Then the induced compression map
\[C_{\mathcal K_0}: \pi(C^*(I,\mathcal T(P))) \to \{P_{\mathcal K_0} U|_{\mathcal K_0}\}'\subseteq\mathcal B(\mathcal K_0),\quad  X\mapsto P_{\mathcal K_0}X|_{\mathcal K_0} \]
is a unital complete isometry and therefore a $C^*$-isomorphism by a result of Kadison (\cite{Kadison}). Hence by the minimality of the Stinespring representation $\pi$ we have
$\mathcal K=\mathcal K_0$ and therefore $\pi(C^*(I,\mathcal T(P))) =\{U\}'$. This completes the proofs of ${\bf{ P_1}}$, ${\bf{ P_2}}$ and ${\bf{ P_3}}$.

To prove ${\bf{ P_4}}$, first note that $\Theta$ is completely contractive and unital as $\pi(Q)=I$. We have also proved that
$\Theta (X)\fJ=\fJ X$ for all $X\in\{P\}'$. Since, for $X,Y\in \{P\}'$,
\[ \delta(\Theta(XY)-\Theta(X)\Theta(Y))=\fJ^*\fJ XY-\fJ^*\Theta(X)\Theta(Y)\fJ=0,\]
then by injectivity of $\delta$, we have $\Theta$ to be multiplicative. Hence ${\bf{ P_4}}$ is proved. \end{proof}

With these properties of the Stinespring dilation of $\Phi$ at hand, we start the proof of the theorem. Define
\begin{align*}
    R_i := \pi (QS_i) \mbox{ for } 1\leq i \leq d-1 \quad\text{and}\quad U=\pi(QP).
\end{align*}
We first show that $(\mathcal K , \mathsf R)$ is a unitary module. To that end, we shall use Theorem 4.2 of \cite{BSR}. Let $\gamma_i = (d - i)/d$ for $i=1,2, \ldots , d-1$. Since $\Theta$ in $({\bf{ P_4}})$ is multiplicative, the tuple $(R_1,\dots,R_{d-1},U)$ is commuting. Note that $\pi\circ\Phi(X)=\pi(X)$, for every $X\in\mathcal C^*(I_\mathcal H, \mathcal T(P))$,
which follows from the facts that $\rho\circ\pi(X)=\Phi(X)$ for all $X\in\mathcal C^*(I_\mathcal H, \mathcal T(P))$ and $\pi\circ \rho= I$.

Also note that for each $1\leq i \leq d-1$,
$$R_i^*U=\pi(S_i^*Q^2P)=\pi(\Phi({S_i}^*Q^2P))=\pi({S_i}^*\Phi(Q^2)P)=\pi(S_i^*QP)=\pi(QS_{d-i})=R_{d-i}.$$
It remains to show that the tuple $(\gamma_1R_1,\dots,\gamma_{d-1}R_{d-1})$ is a $\Gamma_{d-1}$-contraction. Since $(S_1,\dots,S_{d-1},P)$ is a $\Gamma_d$-contraction,
by Lemma 2.7 in \cite{BSR}, $(\gamma_1S_1,\dots,\gamma_{d-1}S_{d-1})$ is a $\Gamma_{d-1}$-contraction. Since $\Theta$ in $({\bf{ P_4}})$ is multiplicative, $(\gamma_1R_1,\dots,\gamma_{d-1}R_{d-1})$ is also a $\Gamma_{d-1}$-contraction.
It follows that the tuple $(R_1,\dots,R_{d-1},U)$ is a $\Gamma_d$-unitary.

That $(\mathcal H , \mathsf S)$ is canonically embedded as a submodule of the unitary module $(\mathcal K , \mathsf R)$ by $\fJ$ follows from the operator identity
 $\pi(QX)\fJ =\fJ X$ for every $X\in \{P\}'$ which was proved in the paragraphs following \eqref{thetaXVisVX}.

Minimality follows from $({\bf{P_1}})$, which says that $\cK$ is actually equal to
\begin{align*}
\overline{\operatorname{span}}\{U^m\fJ h:h\in\cH \text{ and }m\in \mathbb Z\}.
\end{align*}
Let $\rho$ be as in $(\bf P_2)$ above. Consider the restriction of $\rho$ to $\{R_1,\dots,R_{d-1},U\}'$ and continue to denote it by $\rho$. Since complete isometry is a hereditary property, to prove part (1),
 all we have to show is that $\rho(Y)$ lands in $\cT(\mathsf S)$, whenever $Y$ is in $\{R_1,\dots,R_{d-1},U\}'$ and $\rho$ is surjective.
 To that end, let $Y \in \{R_1,\dots,R_{d-1},U\}'$.  Then for each $i=1,2,\dots,d-1$, we see that
\begin{eqnarray*}
S_i^*\rho(Y)P&=&S_i^*\fJ^*Y\fJ P=\fJ^*R_i^*YU\fJ=\fJ^*R_{d-i}Y\fJ\\
&=&\fJ^*YR_{d-i}\fJ=\fJ^*Y\fJ S_{d-i}=\rho(Y)S_{d-i}.
\end{eqnarray*}
Thus $\rho$ maps $\{R_1,\dots,R_{d-1},U\}'$ into $\mathcal T(\mathsf S)$. For proving surjectivity of $\rho$, pick $X\in \mathcal T (\mathsf S)$. Then by $({\bf{ P_2}})$ above there exists a $Y$ in $\{U\}'$ such that $\rho(Y)=\fJ^*Y\fJ=X$.
We have to show that $Y$ commutes with each $R_i$. Since $X$ is in $\mathcal T (\mathsf S)$, we have $S_i^*\fJ^*Y\fJ P=\fJ^*Y\fJ S_{d-i}$.
Therefore by the intertwining property of $\fJ$ we have $\fJ^*R_i^*YU\fJ=\fJ^*YR_{d-i}\fJ$, which is the same as $\fJ ^*R_{d-i}^*Y\fJ=\fJ^*YR_{d-i}\fJ$.
This implies that for each $i=1,2,\dots, d-1$, $$\rho(YR_{d-i}-R_{d-i}Y)=\fJ^*(YR_{d-i}-R_{d-i}Y)\fJ=0.$$
This establishes the commutativity of $Y$ with each $R_i$, since $\rho$ is an isometry. This completes the proof of part (1).

Part (2) of the theorem follows from the content of $({\bf{P_3}})$ if we restrict $\pi$ to ${C^*(I,\mathcal T(\mathsf S))}$ and continue to call it $\pi$.

For the last part of theorem, let us take $\Theta$ as in $({\bf{P_4}})$, i,e.,
$$\Theta(X)=\pi(QX)$$ for every $X$ in $\{P\}'$. Restrict $\Theta$ to ${\{S_1,\dots,S_{d-1},P\}'}$ and continue to call it $\Theta$. The aim is to show that $\Theta(X)\in \{R_1, \dots ,R_{d-1},U\}'$ if $X\in \{S_1,\dots,S_{d-1},P\}'$.
For this we first observe that if $X$ commutes with each $S_j$, then $QX$ is in $\cT(\mathsf S)$.
Now the rest of the proof follows from part (2) of the theorem and \eqref{thetaXVisVX}.
\qed

\subsection{Proof of Theorem \ref{Thm:Gamma_n isometry}.}

 Let $Q$, $\fJ$, $\pi$, $\rho$ and $\mathsf R= (R_1\dots,R_{d-1}, U)$ be as in Theorem \ref{Thm:ComLftPDisc}. We first note that $\fJ$ is an isometry because  $\fJ^*\fJ = Q=\text{SOT}-\lim_j P^{* j}P^j=I$. We shall identify the module $(\mathcal H , \mathsf S)$ with the submodule $(\fJ\mathcal H , \fJ\mathsf S \fJ^*$) of $(\mathcal K, \mathsf R)$. Thus, we get from part $(1)$ of Theorem \ref{Thm:ComLftPDisc} that $(\mathcal K, \mathsf R)$ is a minimal unitary extension of $(\mathcal H , \mathsf S)$. Now let $X$ be an operator on $\mathcal H$ which commutes with $\mathsf S$. Set $Y:=\pi(X)$. Then by part $(3)$ of Theorem \ref{Thm:ComLftPDisc}, it follows that $Y$ commutes with $\mathsf R$ and $Y |_{\mathcal H} = X$, that is $Y$ is an extension of $X$.  Also since the map $\rho$, as in part (2) of Theorem \ref{Thm:ComLftPDisc} is an isometry, $Y$ is a unique norm preserving extension of $X$.
  Thus part (1) follows.

 Part (2) follows straightforward by setting $Y:=\pi(X)$ and remembering that $\pi$ is the minimal Stinespring dilation of the CP map $\Phi_0=\Phi|_{\mathcal C^*(I_{\mathcal H},\mathcal T (P))}: \mathcal C^*(I_{\mathcal H},\mathcal T (P))\to \mathcal B(\mathcal H)$ where $\Phi: \mathcal B(\mathcal H)\to \mathcal{B}(\mathcal H)$ is an idempotent CP map with range $\mathcal T (P)$.

 To prove part (3), set $\pi_0$ to be the restriction of $\pi$ to $\mathcal C^*(\mathsf S)$. The representation $\pi$
maps the generating set $\mathsf S$ of $\mathcal C^*(\mathsf S)$ to the generating set $\mathsf R$ of $\mathcal C^*(\mathsf R)$.
Since $\pi_0(\mathsf S)=\mathsf R$, the range of $\pi_0$ is $\mathcal C^*(\mathsf R)$.
Therefore to prove that the following sequence
$$
0\rightarrow\mathcal I(\mathsf S)\hookrightarrow \mathcal C^*(\mathsf S)\xrightarrow{\pi_0} \mathcal C^*(\mathsf R)\rightarrow 0
$$
is a short exact sequence, all we need to show is that ker$\pi_0=\mathcal I(\mathsf S)$.

Since $\pi_0(\mathcal C^*(\mathsf S))$ is abelian, we have $XY-YX$ in the kernel of $\pi_0$, for any $X,Y\in \mathcal C^*(\mathsf S)\cap \mathcal T(\mathsf S)$.
Hence $\mathcal I(\mathsf S)\subseteq$ ker$\pi_0$. To prove the other inclusion, let $Z_1$ be a finite product of members of $\mathsf S^*$ and $Z_2$ be a finite product of members of $\mathsf S$ and call $Z=Z_1Z_2$.
Since $Z\in\mathcal T (\mathsf S)\subseteq\mathcal T (P)$, we have $\Phi_0(Z)=Z$. Note that $\Phi_0(Z)=P_{\mathcal H}\pi_0(Z)|_{\mathcal H}$,
for every $Z\in \mathcal C^*(\mathsf S)$. Now let $Z$ be any arbitrary finite product of members from $\mathsf S$ and $\mathsf S^*$.
Since $\pi_0(\mathsf S)=\mathsf R$, which is a family of normal operators, we obtain, by the Fuglede-Putnam Theorem that,
action of $\Phi_0$ on $Z$ has all the members from $\mathsf S^*$ at the left and all the members from $\mathsf S$ at the right.
It follows from $\text{ker}\pi=\text{ker}\Phi_0$ and idempotence of $\Phi_0$ that ker$\pi_0=\{Z-\Phi_0(Z):Z\in \mathcal C^*(\mathsf S)\}$.

Because of the above action of $\Phi_0$, if $Z$ is a finite product of elements from $\mathsf S$ and $\mathsf S^*$ then a simple commutator manipulation shows that $Z-\Phi_0(Z)$ belongs to the ideal generated by all the commutators $XY-YX$, where $X,Y \in \mathcal C^*(\mathsf S)\cap \mathcal T(\mathsf S)$. This shows that ker$\pi_0=\mathcal I(\mathsf S)$.

In order to find a completely isometric cross section, set $\rho_0:=\rho|_{\pi(\mathcal C^*(\mathsf S))}$.
Then by the definition of $\rho$ and the action of $\Phi_0$, it follows that $\rho_0(\pi(X))=\fJ^*\pi(X)\fJ=\Phi_0(X)\in \mathcal C^*(\mathsf S)$ for all $X\in \mathcal C^*(\mathsf S)$.
Thus $\operatorname{\operatorname{Ran} } \rho_0\subseteq \mathcal C^*(\mathsf S)$ and therefore is a
completely isometric cross section.
This completes the proof of the theorem.

\section{An application}

In this section, we prove Theorem \ref{Thm:PLT}.
Let $Q$ be the limit as in \eqref{assymplimit}. Consider the isometric module $(\overline{\operatorname{Ran}}Q, \mathsf T)$ constructed in the proof of Theorem~\ref{Thm:Ext}, see equations  \eqref{X} and \eqref{Xj}. We shall obtain a bounded operator $\tilde X$ acting on $\overline{\operatorname{Ran}}Q$ with the following properties:

\begin{enumerate}
\item $\tilde X$ would commute with $\mathsf T=(T_1,\dots, T_{d-1}, V)$ and
\item $\| \tilde X \| \le \|X\|$.
\end{enumerate}
We shall then apply the commutant extension theorem established in part (1) of Theorem \ref{Thm:Gamma_n isometry}.

To that end, we first do a simple inner product computation. For every $h\in\cH$
\begin{align*}
\|Q^\frac{1}{2}Xh\|^2=\langle X^*QXh,h\rangle=\lim_n\langle P^{*n}X^*XP^nh,h \rangle\leq\|X\|^2\langle Qh,h\rangle.
\end{align*}
Thus there is a bounded operator $\tilde X:\overline{\operatorname{Ran }} Q\to\overline{\operatorname{Ran }} Q$ such that
$$
\tilde X:Q^\frac{1}{2}h\mapsto Q^\frac{1}{2}Xh.
$$with norm at most $\|X\|$. Let $j=1,2,\dots,d-1$ and let $T_j$ be the operators defined in \eqref{Xj}. Then for each $h\in\cH$, we have
\begin{align*}
\tilde X T_jQ^\frac{1}{2}h=\tilde X Q^\frac{1}{2}S_jh=Q^\frac{1}{2}XS_jh=Q^\frac{1}{2}S_jXh=T_jQ^\frac{1}{2}Xh=T_j\tilde XQ^\frac{1}{2}h
\end{align*}
showing that $\tilde X$ commutes with $T_j$ for all $j=1,\dots,d-1$. A similar computation also establishes that $\tilde X V= V\tilde X$. We are now ready to apply the technique of commutant extension mentioned above.

Consider $\mathsf R= (R_1,\dots,R_{d-1}, U)$ acting on $\cK$ as in \eqref{Int}. Recall that $(\cK, \mathsf R)$ is a unitary module and there is a contraction $\fJ:\cH\to\cK$ defined as $\fJ h=Q^\frac{1}{2}h$ such that \eqref{Int} holds. Now by the moreover part of item (1) in Theorem~\ref{Thm:Gamma_n isometry}, there exists an operator $Y$ in the commutant of $\mathsf R$ such that $Y|_{\overline{\operatorname{Ran}}Q}=\tilde X$ and $\|Y\|=\|\tilde X\|\leq \|X\|$. Finally, we note that for every $h\in\cH$,
\begin{align*}
\fJ Xh=Q^\frac{1}{2}Xh=\tilde XQ^\frac{1}{2}h=YQ^\frac{1}{2}h=Y\fJ h.
\end{align*}This completes the proof.

The following analogue of the intertwining lifting theorem is easily obtained as a corollary to Theorem \ref{Thm:PLT} by a $2 \times 2$ operator matrix trick.

\begin{corollary}
Let $(\cH, \mathsf S)$ and $(\cH', \mathsf S')$ be two contractive modules. Let $(\fJ,\cK,\mathsf R)$ and $(\fJ',\cK',\mathsf R')$ be as obtained by part (3) of Theorem \ref{Thm:Ext} corresponding to $(\cH, \mathsf S)$ and $(\cH', \mathsf S')$ respectively. Then corresponding to any operator $X:\cH\to\cH'$ intertwining $\mathsf S$ and $\mathsf S'$ there exists another operator $Y:\cK\to\cK'$ such that $Y$ intertwines $\mathsf R$ and $\mathsf R'$, $Y\fJ=\fJ' X$ and $\|Y\|\leq \|X\|$.
\end{corollary}

\vspace{0.1in} \noindent\textbf{Acknowledgement:}
The research work of the first named author is supported by a J C Bose National Fellowship JCB/2021/000041 and that of the second and third named authors are supported by DST-INSPIRE Faculty Fellowships DST/INSPIRE/04/2015/001094 and DST/INSPIRE/04/2018/002458 respectively.

\end{document}